\documentclass[11pt, reqno]{amsart}

\usepackage[T1]{fontenc}
\usepackage{amsthm, amsfonts, mathrsfs}
\usepackage{amssymb, amsmath, theoremref}
\usepackage{times, txfonts, pifont, graphicx}

\theoremstyle{plain}
\newtheorem{claim}{\sc Claim}[section]
\newtheorem{corollary}[claim]{\sc Corollary}

\newtheorem{theorem}[claim]{\sc Theorem}

\theoremstyle{definition}

\theoremstyle{remark}

\newcommand{\integer}{\mathbf{Z}}

\begin{document}

\title[better than nice geometry]{A geometry where everything is
  better than nice} 
\author[L. Bates and P. Gibson]{Larry Bates and
  Peter Gibson}

\date{\today }

\begin{abstract}
We present a geometry in the disk whose metric truth is curiously arithmetic.
\end{abstract}

\maketitle

\section{The better-than-nice metric}
Consider the metric 
\[
  g=\frac{4}{1-r^2}(dx^2 +dy^2)
\]
in the unit disk.  Here $r^2=x^2+y^2$ as usual.  Everything of interest can be computed explicitly, and with surprising results.

\subsection{Hypocycloids in the disk}

Consider the curve
\[
  c(t) = (1-a)e^{i\theta(t)} +ae^{-i\phi(t)}, \qquad 0<a<1,
\]
thought of as a point on a circle of radius $a$ turning at a rate
$\dot{\phi}$ in the clockwise direction as the centre of the circle
rotates on a circle of radius $1-a$ rotating at a rate $\dot{\theta}$
in the counterclockwise direction.

For the small circle of radius $a$ to roll without slipping on the inside 
of the circle of radius $1$ requires the point $c(t)$ to have velocity 
$0$ when $|c(t)|=1$, which is the relation
\[
  a\dot{\phi}=(1-a)\dot{\theta}.
\]
Because $g$ is rotationally invariant, without loss of generality we
may take 
\[
\theta(t) = \frac{at}{2\sqrt{a(1-a)}},\qquad\phi(t)=\frac{(1-a)t}{2\sqrt{a(1-a)}}.
\]

\begin{theorem}
The curve $c(t)$ is a geodesic for the metric $g$, parameterized by arclength.
\end{theorem}
\begin{proof}
For our $g$, the equations $\ddot{u}^i+\Gamma^i_{jk}\dot{u}^j\dot{u}^k=0$ determining geodesics $u(t)=\bigl(x(t),y(t)\bigr)$ parameterized proportional to  arclength are 
\begin{equation}
\begin{split}\nonumber
(1-x^2-y^2)\ddot{x}+x(\dot{x}^2-\dot{y}^2)+2y\dot{x}\dot{y}&=0\quad \mbox{and}\\(1-x^2-y^2)\ddot{y}-y(\dot{x}^2-\dot{y}^2)+2x\dot{x}\dot{y}&=0,
\end{split}
\end{equation}
which can be expressed in terms of $z=x+iy$, as the single equation
\begin{equation}\label{z}
(1-z\bar{z})\ddot{z}+\bar{z}\dot{z}^2=0.
\end{equation}
Given the formula for $c(t)$ it is straightforward to verify that $c(t)$ satisfies (\ref{z}) and moreover that $|\dot{c}|=\sqrt{1-|c|^2}/2$, from which it follows that $\lVert\dot{c}\rVert_g=1$, meaning that the parameterization is by arclength. 
\end{proof}
\begin{theorem}
The closed geodesics (i.e. keep rolling the generating circle of the hypocycloid until it closes up) have  length $4\pi\sqrt{n}$, and the number of geometrically distinct geodesics of length $4\pi\sqrt{n}$ is given by the arithmetic function $\psi(n)$. 
\end{theorem}
The function $\psi(n)$ counts the number of different ways that the integer $n$ may be written as a product $n=pq$, with $p\leq q$, $(p,q)=1$. Values of this function are tabulated in sequence $A007875$ in the online encyclopedia of integer sequences \cite{OEIS}.
\begin{proof}
Beginning at $c(0)=1$, the geodesic $c(t)$ first returns to the boundary circle at 
\[
c\,\Bigl(4\pi\sqrt{a(1-a)}\Bigr)=e^{2\pi i(1-a)},
\]
returning again at points of the form $e^{2\pi im(1-a)}$ $(m\in\integer_+)$.  The corresponding succession of cycloidal geodesic arcs winds clockwise around the origin if $0<a<1/2$ and counterclockwise if $1/2<a<1$; when $a=1/2$, $c(t)$ traverses back and forth along the $x$-axis.  Thus $c(t)$ forms a once-covered  closed geodesic precisely when $2\pi m(1-a)=q2\pi$ for some relatively prime pair of positive integers $q<m$, in which case the geodesic has length 
\[
4\pi m\sqrt{a(1-a)}=4\pi\sqrt{(m-q)q}.
\]
To count geometrically distinct closed geodesics we restrict to $0<a\leq1/2$, in which case $p:=m-q=ma\leq m(1-a)=q$.  Given any relatively prime positive integers $p\leq q$ the geodesic of length $4\pi\sqrt{pq}$ occurs when $a=p/(p+q)$.  
\end{proof}

\subsection{Eigenfunctions and eigenvalues of the Laplacian}

Set the Laplacian $\Delta$ to be
\[
  \Delta = -g^{ab}\nabla_a\nabla_b
\]
where $\nabla_a$ is the covariant derivative operator associated to the metric 
$g$ via the Levi-Civita connection. Consider the eigenvalue problem
\[
  \Delta u = \lambda u
\]
for functions $u$ with the boundary value $u(r=1)=0$.  This problem has a 
number of remarkable features.

\begin{theorem} The eigenfunctions and eigenvalues satisfy
  \begin{enumerate}
    \item The eigenvalues $\lambda_n$ are precisely the positive integers 
$n=1,2,3,\dots$.
    \item The eigenfunctions are polynomials.
    \item The dimension of the eigenspace for eigenvalue $n$ is the number of 
divisors of $n$. (The number of divisors function is denoted by $\tau(n)$.)
  \end{enumerate}
  \end{theorem}
\begin{proof}
Since the operator
\[
-\Delta+\lambda=\frac{1-r^2}{4}\left(\frac{\partial^2}{\partial x^2}+\frac{\partial^2}{\partial y^2}\right)+\lambda
\]
is analytic hypoelliptic, distributional solutions to the eigenvalue equation $\Delta u=\lambda u$ are necessarily real analytic, and representable near $(x,y)=(0,0)$ by absolutely convergent Fourier series
\[
u(r,\theta)=\sum_{n\in\integer}a_n(r)e^{in\theta}.
\]
Observing that $(-\Delta+\lambda)\bigl(a_n(r)e^{in\theta}\bigr)$ has the form $A_n(r)e^{in\theta}$ for some $A_n$, it follows that $u$ is an eigenfunction of $\Delta$ only if each summand $a_n(r)e^{in\theta}$ is, so it suffices to consider just products of the form
\[
u(r,\theta)=f(r)e^{in\theta}.
\]
Expressing $-\Delta+\lambda$ in polar coordinates yields
\[
(-\Delta+\lambda)\bigl(f(r)e^{in\theta}\bigr)=\frac{1-r^2}{4}\left(f^{\prime\prime}(r)+\frac{1}{r}f^\prime(r)+\left(\frac{4\lambda}{1-r^2}-\frac{n^2}{r^2}\right)f(r)\right)e^{in\theta}.
\]
Therefore $u(r,\theta)=f(r)e^{in\theta}$ is an eigenfunction of  $\Delta$ only if $f$ satisfies 
\[
r^2(1-r^2)f^{\prime\prime}+r(1-r^2)f^\prime+\bigl(4\lambda r^2-n^2(1-r^2)\bigr)f=0.
\]
Assume for definiteness that $n\geq0$ and write $f(r)=r^ng(r^2)$, so that $g$ satisfies the hypergeometric equation
\[
r(1-r)g^{\prime\prime}+(c-(a+b+1)r)g^\prime-ab\,g=0,
\]
where 
\[
a=\Bigl(n+\sqrt{n^2+4\lambda}\Bigr)/2,\quad b=\Bigl(n-\sqrt{n^2+4\lambda}\Bigr)/2,\mbox{ and }c=n+1.
\]
This has a unique non-singular solution (up to scalar multiplication), the hypergeometric function
\[
g(r)=_2\!F_1(a,b;n+1;r)\quad\mbox{ where }\quad g(1)=\frac{n!}{\Gamma(1+a)\Gamma(1+b)},
\]
which is zero at $r=1$ if and only if $b=-m$ for some integer $m\geq 1$.  This implies $\lambda=m(m+n)$ is a positive integer, proving part (1) of the theorem.  If $n>0$ there are two corresponding eigenfunctions $u(r,\theta)=r^ng(r^2)e^{\pm in\theta}$, making a total of $\tau(\lambda)$ eigenfunctions for each positive integer eigenvalue $\lambda$.  That these are linearly independent (part (3)), and that the eigenfunctions are polynomials (part (2)) can be verified using an explicit formula for the eigenfunctions, as follows. 

Using the formulation $\Delta=-(1-z\overline{z})\frac{\partial^2}{\partial\overline{z}\partial z}$, where $z=x+iy$, one can check directly that the Rodrigues-type formula
\begin{equation}\label{Rodrigues}
 u^{(p,q)}(z) = \frac{(-1)^p}{q(p+q-1)!}(1-z\bar{z})\frac{\partial^{p+q}}{\partial\bar{z}^p\partial z^q} (1-z\bar{z})^{p+q-1}
\end{equation}
represents eigenfunctions corresponding to $\lambda=pq$, i.e., $\Delta u^{(p,q)}=pq\,u^{(p,q)}$.  
\end{proof}
\subsection{Two corollaries}

\begin{corollary}
The spectral function is precisely the square of the Riemann zeta function
\[
\sum_n\frac{1}{(\lambda_n)^s} = \sum_n \frac{\tau(n)}{n^s} = (\zeta(s))^2.
\]
\end{corollary}

In a more applied vein, consider a unit radius circular membrane fixed at the boundary (i.e. a drumhead) having constant tensile force per unit length $S$ and radially varying density $\rho(r)=4S/(1-r^2)$. Small transverse displacements of the membrane $w(x,y,t)$ are governed by the equation 
\begin{equation}\label{harmonic}
w_{tt}+\Delta w=0,
\end{equation}
so squared eigenfrequencies of the membrane correspond to eigenvalues of $\Delta$. 
\begin{corollary}
Unlike the standard vibrating membrane whose eigenfrequencies are proportional to zeros of $J_0$, for each eigenfrequency $\omega_n=\sqrt{n}$ of (\ref{harmonic}), all the higher harmonics $m\omega_n$ $(2<m\in\integer_+)$ are also eigenfrequencies. 
\end{corollary}

\subsection{Acoustic imaging and combinatorics}

Supplemented by the sequence of monomials $u^{(p,0)}(z):=z^p$ for $p\geq 0$, the eigenfunctions of $\Delta$ are the special functions suited to acoustic imaging of layered media. 
A layered medium consists of a stack of $n-1$ horizontal slabs between two semi-infinite half spaces, where each slab and half space has constant acoustic impedance. Thus impedance as a function of depth is a step function having jumps at the $n$ interfaces. To image the layers, an impulsive horizontal plane wave is transmitted at time $t=0$ from a reference plane in the upper half space down toward the stack of horizontal slabs, and the resulting echoes are recorded at the reference plane, producing a function $G(t)$ $(t>0)$ (the boundary Green's function of the medium).  Let $L_1$ denote the time required for acoustic waves to travel from the reference plane to the first interface and back, with $L_2,\ldots,L_n$ denoting two-way travel time within each successive layer. When a wave travels downward toward the $j$th interface, it is partly reflected back, with amplitude factor $R_j$, and partly transmitted, with transmission factor $1-R_j$ $(1\leq j\leq n)$. 
%Let $R_1,\ldots,R_n$ denote the reflection coefficients at the $n$ interfaces.  
%Let $G(t)$ denote the normal incidence plane wave impulse response at the boundary of an $n$-layered medium having layer thicknesses $L=(L_1,\ldots,L_n)$ (in two-way travel time) and reflection coefficients at layer interfaces $R=(R_1,\ldots,R_n)\in(-1,1)^n$. 
\begin{theorem}
\[
G(t)=\sum_{k\in\{1\}\times\integer_+^{n-1}}\left(\prod_{j=1}^nu^{(k_j,k_{j+1})}(R_j)\right)\delta\bigl(t-\langle L,k\rangle\bigr).
\]
\end{theorem}
Here $L=(L_1,\ldots,L_n)$, $u^{(0,q)}\equiv0$ if $q\geq1$ and $k_j=0$ if $j>n$.  
\begin{proof}
Expanding the binomial $(1-z\bar{z})^{p+q-1}$ in the formula (\ref{Rodrigues}), and then applying the derivative $\rule{0pt}{13pt}\partial^{p+q}/\partial\bar{z}^p\partial z^q$, yields
\[u^{(p,q)}(z)=
\frac{(-1)^{q+\nu+1}}{q}(1-z\bar{z})z^{m+\nu-q+1}\bar{z}^{m+\nu-p+1}\sum_{j=0}^{\nu}(-1)^j\frac{(j+\nu+m+1)!}{j!(j+m)!(\nu-j)!}(z\bar{z})^j,
\]
where $m=|p-q|$ and $\nu=\min\{p,q\}-1$.  Switching to polar form $z=r e^{i\theta}$, it follows that 
\[
u^{(p,q)}\bigl(r e^{i\theta}\bigr)=e^{i(p-q)\theta}\frac{(-1)^{q+\nu+1}}{q}(1-r ^2)r ^m\sum_{j=0}^{\nu}(-1)^j\frac{(j+\nu+m+1)!}{j!(j+m)!(\nu-j)!}r ^{2j}.
\]
For $\theta=0,\pi$, the latter coincide with the functions $f^{(p,q)}$ occurring in \cite[Thm.~2.4, 4.3]{Gi:SIAP2014}.  
\end{proof}
Each term $\left(\prod_{j=1}^nu^{(k_j,k_{j+1})}(R_j)\right)\delta\bigl(t-\langle L,k\rangle\bigr)$ corresponds to the set of all scattering sequences that have a common arrival time $t_i=\langle L,k\rangle$, with each individual scattering sequence weighted according to the corresponding succession of reflections and transmissions.  A scattering sequence may be represented by a Dyck path as in Figure~\ref{fig}.  
The tensor products $\prod_{j=1}^nu^{(k_j,k_{j+1})}(R_j)$ thus have a combinatorial interpretation in that they count weighted Dyck paths having $2k_j$ edges at height $j$. See \cite{Gi:SIAP2014}. 
\begin{figure}[h]
\fbox{
\includegraphics[clip,trim=0in 2in 0in -2in, width=3in]{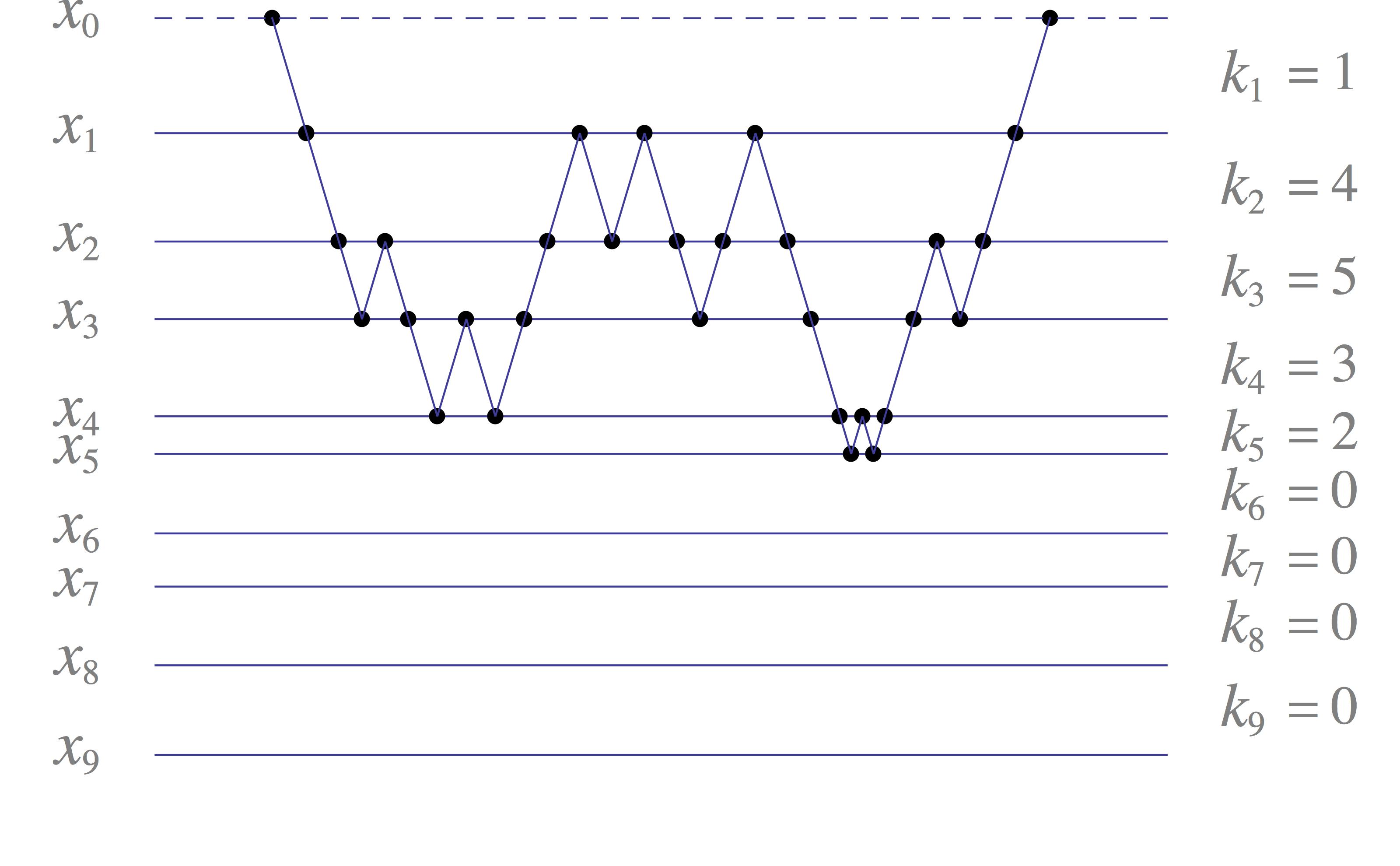}%left,bottom,right,top
}
\caption{ The Dyck path for a scattering sequence. Time increases to the right. Each node at depth $x_j$ receives a weight according to the structure of the path at the node: weights $R_j,-R_j,1-R_j,1+R_j$ correspond respectively to down-up reflection, up-down reflection, downward transmission, upward transmission. The scattering sequence returns to the reference depth $x_0$ at time $t_i=\langle L,k\rangle$, where $k=(k_1,\ldots,k_n)$. Its amplitude is the product of the weights.}\label{fig}
\end{figure}

\vspace{20pt}

\noindent Larry M. Bates \newline
Department of Mathematics \newline
University of Calgary \newline
Calgary, Alberta \newline
Canada T2N 1N4 \newline
bates@ucalgary.ca\newline

\vspace{20pt}

\noindent Peter C. Gibson \newline
Department of Mathematics \newline
York University\newline
Toronto, Ontario\newline
Canada M3J 1P3\newline
pcgibson@yorku.ca 

\end{document}